\documentclass[12pt]{amsart}
\pdfoutput=1

\usepackage{amsmath, amssymb, amsthm, amsfonts, amscd}

\input xy
\xyoption{all}

\usepackage{hyperref}
\usepackage{graphicx}
\usepackage{color}

\usepackage[draft]{say}
\definecolor{egyptianblue}{rgb}{0.06, 0.2, 0.65}

\usepackage[margin=1in,marginparwidth=0.8in, marginparsep=0.1in]{geometry}

\usepackage{soul}

\usepackage[all, knot, poly]{xy}
\xyoption{knot}

\usepackage{times}
\usepackage{mathrsfs}

\newtheorem{lemma}{Lemma}
\newtheorem{proposition}[lemma]{Proposition}
\newtheorem{corollary}[lemma]{Corollary}
\newtheorem{theorem}[lemma]{Theorem}

\theoremstyle{definition}

\theoremstyle{remark} 
\newtheorem*{remark}{Remark}

\newcommand{\A}{\mathbb{A}}

\newcommand{\C}{\mathbb{C}}
\newcommand{\D}{\mathbb{D}}

\newcommand{\G}{\mathbb{G}}

\newcommand{\R}{\mathbb{R}}

\newcommand{\Z}{\mathbb{Z}}

\DeclareMathOperator{\Hom}{Hom}

\DeclareMathOperator{\Mod}{Mod}
\DeclareMathOperator{\Coh}{Coh}
\DeclareMathOperator{\Perf}{Perf}
\DeclareMathOperator{\Sh}{Sh}
\DeclareMathOperator{\Fuk}{Fuk}

\DeclareMathSymbol{\shortminus}{\mathbin}{AMSa}{"39}

\usepackage{quiver}

\usepackage{epigraph}

\begin{document}

\title{Toric mirror monodromies and Lagrangian spheres}

\author{Vivek Shende}

\begin{abstract}
The central fiber of a Gross-Siebert type toric degeneration is known to
satisfy homological mirror symmetry: its category of coherent sheaves
is equivalent to the wrapped Fukaya category of a certain exact symplectic manifold.  Here we show that, in the Calabi-Yau case, the images of line bundles are represented by  Lagrangian spheres.  
\end{abstract}

\maketitle

\epigraph{And their appearance, and their work, was as it were a wheel in the middle of a wheel.}{Ezekiel 1:16}

\section{Introduction}

Let $X$ be a Weinstein symplectic manifold, i.e. an exact symplectic manifold obtained from a standard ball by iterated handle attachment along Legendrian spheres.  These handle attachments come with `co-core' disks, which are known to generate the wrapped Fukaya category $\Fuk(X)$ \cite{CDGG, GPS2}.

Let $L \subset X$ be a compact exact Lagrangian, possibly carrying a finite rank local system; or, more generally, possibly immersed and carrying the appropriate structures per \cite{akaho-joyce}.  It follows from the definitions that for any, possibly noncompact, $T \subset X$, the Hom space $\Hom(T, L)$ is finite rank.  We write $\Fuk_{cpt}(X)$ for the idempotent completion of the full subcategory generated by such compact Lagrangians.  More generally, we say some $M \in \Fuk(X)$ is {\em pseudo-perfect} if always $\Hom(T, M)$ is finite rank.  We write $\Fuk(X)^{pp}$ for full subcategory on such objects.\footnote{It is more typical to use this notation of the category of pseudo-perfect modules;  since $\Fuk(X)$ is known to be smooth \cite{ganatra-duality}, all pseudo-perfect modules are representable and so there is no ambiguity.} 
That is, we have an inclusion $\Fuk_{cpt}(X) \subset \Fuk(X)^{pp}$.  
It is expected, but not known, that this inclusion is an equality.  When this holds, we say $X$ has {\em enough compact Lagrangians}.  

One can ask similar questions for Lagrangians with prescribed topology; 
we write $\Fuk_{cemb}(X)$ for the subcategory  split-generated by 
compact embedded Lagrangians carrying finite rank local systems, and $\Fuk_{sph}(X)$ for the subcategory  split-generated by embedded Lagrangian spheres.  So, 
$$\Fuk_{sph}(X) \subset \Fuk_{cemb}(X) \subset \Fuk_{cpt}(X) \subset \Fuk(X)^{pp}$$
We correspondingly say that $X$ has {\em enough compact embedded Lagrangians} (resp.  {\em enough Lagrangian spheres}) if the composition of the last two (resp. all three) inclusions is an equality.  Note however that these properties are not generally valid for Weinstein manifolds: e.g., it follows from 
\cite{Abouzaid-cotangent} that cotangent bundles have enough compact embedded Lagrangians, but from \cite{Fukaya-Seidel-Smith, Nadler-cotangent} that they do not have enough Lagrangian spheres unless the base manifold was itself a homotopy sphere.  Meanwhile, the results of \cite{dimitroglou-whitneysphere} at least strongly suggest that $\C^2 \setminus \{xy=1\}$ does not have enough compact embedded Lagrangians.

The existence of enough compact Lagrangians is of particular relevance in the context of mirror symmetry.  It is often comparatively easy to prove homological mirror symmetry results for Weinstein $X$, especially now after the localization theorems \cite{GPS2, GPS3}. 
Suppose given some such:
\begin{equation}\label{noncompact mirror}\Fuk(X) = \Coh(Y)\end{equation}
Such $Y$ is necessarily singular or noncompact \cite{ganatra-proper}.  In the present article we will consider a situation where $Y$ is compact (and hence singular).  In such cases, and certainly for those of interest here, it is expected that smoothing $Y$ corresponds under mirror symmetry to compactifying $X$.  That is, one expects a compactification $X \subset \overline{X}$, a smoothing $\widetilde{Y}$ of $Y$, and an isomorphism \begin{equation}\mathrm{Fuk}^?(\overline{X}) \cong \Coh(\widetilde{Y}).\end{equation}  
Here the question mark serves to remind that we remain ambivalent amongst possible definitions of the Fukaya category of a compact symplectic manifold.  But insofar as the objects of such a category are compact Lagrangians in $\overline{X}$, it is only $\Fuk_{cpt}(X)$ which directly deforms to $\Fuk(\overline{X})$.  More forcefully: it is not known to follow from \eqref{noncompact mirror} that $X$ or $\overline{X}$ has any compact Lagrangians at all!  
Note also that when there are enough compact Lagrangians, one can compose identifications
$\Fuk_{cpt}(X) = \Fuk(X)^{pp} = \Coh(Y)^{pp} = \Perf(Y)$.  (Recall we assumed $Y$ is proper.)  

For deeper and more detailed considerations around the topic of the previous paragraph, we refer to the recent proof of homological mirror symmetry for Batyrev mirror pairs \cite{GHHPS}.  Per the authors of that  sizeable work, the ``bulk of the paper'' was concerned with establishing the following result (under slightly more restrictive hypotheses): 

\begin{theorem} \label{fiber enough lagrangians} \cite{GHHPS}
    Let $T$ be a smooth toric DM stack, and $Z_T: (\C^*)^n \to \C$ the mirror Lefschetz fibration.  Then the \cite{Gammage-Shende-very-affine} mirror symmetry between $\partial T$ and the general fiber of $Z_T$ carries line bundles pulled back along $\partial T \to T$ to Lagrangian spheres in $Z_T$.  
\end{theorem}

\begin{corollary}
    $Z_T$ has enough Lagrangian spheres, and in fact, a finite collection split-generate. 
\end{corollary}
\begin{proof}
    It remains to show that $\mathrm{Perf}(\partial T)$ is generated by line bundles. In the nonstacky case: a smooth toric variety $T$ is projective, hence so is $\partial T$, hence finitely many powers of (a restriction of) an ample line bundle will generate $\mathrm{Perf}(\partial T)$ \cite[Thm. 4]{Orlov-dimension}.  More generally,  the argument of loc. cit. shows  that if $\Coh(X)$ is generated by a given collection of vector bundles, and $Y \subset X$ is closed, then $\Perf(Y)$ is generated by pullbacks of these bundles. But for a toric DM stack $T$, it is known that $\Coh(T)$ is generated by line bundles; in fact, an explicit finite collection thereof \cite{Favero-Huang, Hanlon-Hicks-Lazarev}.
\end{proof}

\begin{remark}
    The inputs to the proof of Theorem \ref{fiber enough lagrangians} in  \cite{GHHPS} are that (a) \cite{abouzaid-toric} identifies line bundles with disks whose boundaries are essentially the desired spheres  and (b) \cite{Gammage-Shende-very-affine} explains how toric mirror symmetry passes to the boundary.  One reason for the length of the argument in \cite{GHHPS} is that (a) and (b) are proved in different setups, so some comparison theorems are needed.\footnote{It seems to us that  these comparisons could have been avoided by using \cite{Zhou-ccc} instead of \cite{abouzaid-toric} to identify the images of line bundles under mirror symmetry.  We do not use either of these results.}  Additionally, \cite{GHHPS} needed (c) a result explaining how the pullback of pseudoperfect modules along the covariant inclusion functor of  \cite{GPS1} can sometimes be realized geometrically.
\end{remark}

We will give a proof of Theorem \ref{fiber enough lagrangians} which on the A-side remains entirely within a constructible sheaf model until a final invokation of \cite{GPS3}.  That is, we do not  rely on or have to prove comparisons to the `Floer-theoretic' approach to toric mirror symmetry, i.e. \cite{Abouzaid-cotangent, hanlon-monodromy, Hanlon-Hicks}; as a result, our proof is perhaps less bulky.  Instead we use the fact that our recent account of toric mirror symmetry \cite{shende-toric} naturally included a construction of certain symplectic monodromies.  Below, we check that these realize the autoequivalences mirror to tensoring by line bundles (Proposition \ref{weak intertwining}, Theorem \ref{strong intertwining}).  This reduces Theorem \ref{fiber enough lagrangians} to characterizing the image of $\mathcal{O}$, which in the sheaf formalism is immediate from the monoidality of \cite{FLTZ2} toric mirror symmetry (Lemma \ref{O}).  Passing to Fukaya categories requires a characterization of how the \cite{GPS3} sheaf-Fukaya equivalence acts on compact Lagrangians, which we give in Proposition \ref{viterbo commutes with GPS}. 

\vspace{2mm}

Let us turn to our main goal, which is 
to prove the analogue of Theorem \ref{fiber enough lagrangians} for the mirrors to the central fibers $Y_0$ of Gross-Siebert \cite{Gross-Siebert} toric degenerations $Y_t$.
Such $Y_0$ are built from toric varieties glued together along their boundary strata.  (We will call such structures `toric buildings'.)  Toric boundaries themselves provide some, but certainly not all, examples of toric buildings. 

To organize the mirror symmetry of toric buildings, we introduced in \cite{Gammage-Shende-very-affine}  certain structures called `fanifolds'.  Informally, a fanifold $\Phi$ is a stratified subset $|\Phi|$ of a manifold $\Omega$ together with data equipping the normal geometry to each stratum with the structure of a fan.  We also assume all strata are contractible.  
To such a structure we associated in \cite{Gammage-Shende-very-affine} a toric variety $\mathbf{T}(\Phi)$ and a Weinstein manifold $\mathbf{W}(\Phi)$, and proved the mirror symmetry 
\begin{equation} \label{fanifold mirror} \Coh (\mathbf{T}(\Phi)) = \Fuk(\mathbf{W}(\Phi))
\end{equation}
The basic idea was to glue together toric mirror symmetry.  

We say the fanifold is `closed' when $|\Phi| = \Omega$ is a closed manifold.  
In this case, by gluing together the toric mirror monodromies, we prove here:

\begin{theorem} \label{fanifold enough Lagrangians} Let $\Phi$ be a closed fanifold with simplicial fans.  Then the mirror symmetry \eqref{fanifold mirror} carries line bundles on $\mathbf{T}(\Phi)$ to local systems on embedded compact Lagrangians in $\mathbf{W}(\Phi)$.  The Lagrangians are  diffeomorphic to $|\Phi|$. 
\end{theorem}

In particular, when $\mathbf{T}(\Phi)$ is in addition nonstacky and projective, $\mathbf{W}(\Phi)$ has enough embedded compact Lagrangians.  (Again, we expect this in fact holds more generally.)  When $\mathbf{T}(\Phi)$ arises from a Calabi-Yau degeneration, the corresponding $|\Phi|$ is a sphere, so in this case $\mathbf{W}(\Phi)$ has enough Lagrangian spheres.  When $|\Phi| = \Omega$ is a manifold with boundary, the same argument shows that line bundles are send to local systems on forward-stopped Lagrangians diffeomorphic to $|\Phi|$. 

\begin{remark} The Lagrangians we obtain could likely also be constructed by gluing  the Lagrangians from \cite{abouzaid-toric} or \cite{Zhou-ccc}. \end{remark}

\begin{remark}
    Insofar as Theorem \ref{fiber enough lagrangians} was a key step in proving the mirror symmetry for compact Calabi-Yau hypersurfaces in toric varieties in \cite{GHHPS}, it is natural to ask, given Theorem \ref{fanifold enough Lagrangians}, how far are we from proving mirror symmetry for the Gross-Siebert general fiber?  Two main issues are the following.  First, it has not yet been shown that the space $\mathbf{W}(\Phi)$ is in fact the expected divisor complement in said fiber.  Second, the argument in \cite{GHHPS} seems to rest on the fact that the appropriate equivariantly graded pieces of the symplectic cohomology are either zero or one dimensional \cite[Prop. 6.6]{GHHPS}; I do not know whether an analogous result holds in the Gross-Siebert setup. 
\end{remark}

\vspace{2mm} \noindent {\bf Acknowledgements.}  
The main insights of this article came in the waters of the Black Sea, during the conference ``Complex Geometry at Large'', in Sunny Beach, Bulgaria, 2024.  

I am supported by Villum Fonden Villum Investigator grant 37814, Novo Nordisk Foundation grant NNF20OC0066298, and Danish National Research Foundation grant DNRF157.

\vspace{2mm} {\bf Conventions.} 
The notations $\Coh$, $\Fuk$, $\Sh$, $\mu sh$, etc. always mean the corresponding stable presentable 
$\infty$-categories, i.e. what elsewhere might be denoted $\mathrm{IndCoh}$ or $\Mod \Fuk$, etc.  In particular, any statement asserting generation means what in the symplectic geometry literature is called `split-generation'.

\section{Toric mirror monodromies}

Let us recall that the category of coherent sheaves on a toric variety can be identified with the category of constructible sheaves on a real torus with prescribed microsupport \cite{Bo, FLTZ2, Treumann, FLTZ3, Ku}.  This is a version of mirror symmetry for toric varieties; there are also equivariant and stacky variants of the statements.  

The simplest example is $\A^1/\G_m$.  Let $\Lambda \subset T^* \R$ be the union of the zero section and all negative conormals to integer points,  \cite{FLTZ2} identifies
\begin{align}  \label{equivariant mirror a1}
    \Coh(\A^1/\G_m) & \xrightarrow \sim \Sh_\Lambda(\R) \\
    \mathcal{O}(n) & \mapsto \Z_{(n, \infty)}
\end{align}
Here, $\mathcal{O}(n)$ denotes the structure sheaf twisted by the character $n \in \Z$, and $\Z_{(n, \infty)}$ denotes the extension by zero of the constant sheaf on $(n, \infty)$.  
The matching of Hom spaces is: 
\begin{equation} \label{fltz simplest coh}
\Hom(\mathcal{O}(n), \mathcal{O}(m)) = 
\mbox{degree zero elements of}\, \mathcal{O}(m-n) = \begin{cases} \Z \cdot z^{n-m}  & m \le n \\ 0 & \mbox{otherwise} \end{cases}    
\end{equation}
\begin{equation} \label{fltz simplest sh}
\Hom(\Z_{(n, \infty)}, \Z_{(m, \infty)}) = \Gamma(\D (\Z_{(n, \infty)} \otimes \D \Z_{(m, \infty)})) = \begin{cases} \Gamma (\Z_{[n,\infty)}) = \Z & m \le n \\ 
\Gamma(\Z_{(m, \infty)}) = 0 & \mbox{otherwise}  \end{cases}    
\end{equation}
These collections of objects generate their respective categories, and so the explicit matching above determines an equivalence of categories.  

We will be interested in the collection of autoequivalences 
\begin{eqnarray} 
    \mathcal{P} : \Z & \to & Aut(\Coh(\A^1/\G_m)) \\
    n & \mapsto & \otimes \mathcal{O}(n) 
\end{eqnarray}

We would like to realize this action on the mirror via monodromies.   To this end, consider $\R_x \times \R_t$, and define $\widetilde{\Lambda} \subset \R_x \times \R_t$ to be the union of the zero section and the negative conormals to all integer translates of the line $x = t$.  
This construction descends to $\R_x \times (\R_t/\Z)$.  We observe: 

\begin{lemma}\label{shift as monodromy}
    For any integer $n$, pullback along the inclusion
    $i_n: \R = \R_x \times \{n\} \subset \R_x\times \R_t$ induces an equivalence
    $$i_n^* : \Sh_{\widetilde{\Lambda}}(\R_x \times \R_t) \to \Sh_\Lambda(\R).$$
    The pushforward along $\R_x \times \R_t \to \R_x$ determines a local system of categories on $\R_t$, which descends canonically to $\R_t / \Z$.  The monodromy is calculated by the composition 
    $$\Sh_\Lambda(\R) \xrightarrow{(i_0^*)^{-1}}  \Sh_{\widetilde{\Lambda}}(\R_x \times \R_t) \xrightarrow{i_1^*} \Sh_\Lambda(\R),$$
    and is canonically equivalent to the pushforward along $+1: \R \to \R$. 
\end{lemma}
\begin{proof}
    Obvious after pullback along $(x, t) \mapsto (x, t+x)$. 
\end{proof}

We denote the monodromy as: 
\begin{equation} 
 \widetilde{\mathcal{M}} : \Z \to Aut(\Sh_\Lambda(\R)).
\end{equation}

\begin{lemma} \label{a1 monodromy intertwining} 
    The mirror symmetry \eqref{equivariant mirror a1} intertwines $\mathcal{P}$ and $\widetilde{\mathcal{M}}$. 
\end{lemma}
\begin{proof}
    It is immediate from the definitions
    that \eqref{equivariant mirror a1} intertwines $\mathcal{P}$ with $\Z$-translation on $\R$.  The result now follows from Lemma \ref{shift as monodromy}. 
\end{proof}

In the remainder of this section, we explain a corresponding result for more general toric geometries.  The family of categories with the desired monodromy was  described already in our previous article \cite{shende-toric}.  Let us review the ideas of that article.   Consider a toric stack presented as $(\A^n \setminus Z) / G$ where $G \subset \G_m^n$ and $Z$ is a $G$-invariant union of coordinate planes.  We write $M$ for the character lattice of the residual torus, i.e. 
$$0 \to M \to \Z^n \to \Hom(G, \G_m) \to 0.$$

We described in \cite[Def. 6]{shende-toric} 
a closed subset $\Lambda_{Z} \subset \Lambda^n \subset (T^*\R)^n$; it descends to the quotient $\Lambda_{Z, M} \subset T^*(\R^n/M)$. 
Note that in these conventions, $\Lambda_{\emptyset, 0} = \Lambda^n \subset (T^* \R)^n$.  We showed in \cite[Prop. 10]{shende-toric}:  
\begin{equation} \label{equivariant mirror symmetry} \Xi:  \Coh((\A^n \setminus Z) / G) \cong \Sh_{\Lambda_{Z, M}} ( \R^n / M).\end{equation}
The derivation of \eqref{equivariant mirror symmetry} from \eqref{equivariant mirror a1} proceeds more or less by formal manipulations.  However, the right hand side of \eqref{equivariant mirror symmetry} is {\em not} the category considered by the previous authors \cite{Bo, FLTZ2, Treumann, FLTZ3, Zhou-ccc, Ku}.  To recover the latter, we considered the sequence 
\begin{equation} \label{reduction sequence}
0 \to M_{\R} / M \to \R^n / M \to \R^n / M_{\R} \to 0
\end{equation}
and the corresponding sequence of cotangent bundles.  For $\gamma  \in \R^n / M_\R$, we write 
\begin{equation} \label{at gamma} 
\Lambda_{Z, M, \gamma} \subset T^*(M_{\R} / M)
\end{equation} for the symplectic reduction.  

We proved in \cite{shende-toric} that when $(\A^n \setminus Z) / G$ is Deligne-Mumford (the corresponding fan is simplicial), the restriction of sheaves induces an equivalence 
\begin{equation} \rho_\gamma: \Sh_{\Lambda_{Z, M}} ( \R^n / M) \xrightarrow{\sim} \Sh_{\Lambda_{Z, M, \gamma}}(M_{\R} / M). \end{equation}
When $\gamma = 0$, the right hand side is the category considered previously.  We will write $\Xi_\gamma: = \rho_\gamma \circ \Xi$.  

It follows from the definitions that $\Lambda_{Z, M, \gamma}$ depends only on $\gamma \in \R^n / M_{\R}$ modulo the integer lattice $\Z^n / M$.  Thus the local system of categories over $\R^n / M_{\R}$ with fiber $\Sh_{\Lambda_{Z, M, \gamma}}(M_{\R} / M)$ descends to 
the torus $(\R^n / M_{\R}) / (\Z^n / M)$.  We obtain an action by monodromies: 
\begin{equation} \label{monodromies} 
\mathcal{M}: \Z^n/M \to Aut(\Sh_{\Lambda_{Z, M, 0}}(M_{\R} / M) ).
\end{equation}

Recall the  identification $Pic((\A^n \setminus Z) / G) = \Z^n / M$ \cite{Cox}. This identification can be described in terms of the pullback along
$(\A^n \setminus Z) / G \to B \G_m^n$, which gives a map 
$\Z^n = Pic(B\G_m^n) \to Pic((\A^n \setminus Z) / G)$. 
The map is known to be surjective with kernel $M$.  
Anyway, tensoring by line bundles defines another map 
\begin{equation}
    \mathcal{P}: \Z^n/M = Pic((\A^n \setminus Z) / G) \to Aut(\Coh((\A^n \setminus Z) / G))
\end{equation}
For $\eta \in \Z^n/M$, we will write $\mathcal{O}(\eta) \in Pic((\A^n \setminus Z) / G)$ for the corresponding line bundle. 

\vspace{2mm}
Mirror symmetry often realizes autoequivalences of the derived category of an algebraic variety via monodromies of symplectic manifolds, so it is natural to expect a corresponding intertwining of $\mathcal{M}$ and $\mathcal{P}$.  In fact, a similar statement, for a different construction of the symplectic monodromies and in a different model of toric mirror symmetry, has already appeared in \cite{hanlon-monodromy}.  
In the remainder of this section we construct this intertwining in our setup. 

\vspace{2mm}

\begin{remark}
It would be nontrivial to compute the monodromy action directly from \eqref{monodromies}.  As $\gamma$ varies, the (singular) Legendrians $\partial_\infty \Lambda_{Z, M, \gamma} \subset S^* (M_\R/M)$ are not even homeomorphic, so certainly are not related by some ambient contactomorphism. In particular, this means the monodromy cannot be calculated via the sheaf quantization of \cite{GKS}.    
\end{remark}

Similarly to the spirit of \cite{shende-toric}, we build from the known case of $\A^1/\G_m$.  We may re-do the construction of \eqref{equivariant mirror symmetry} using $\widetilde{\Lambda} \subset T^*(\R_x \times \R_t)$ in place of $\Lambda \subset T^*\R$.  The result is a local system of categories over $\R_t^n$
whose stalk at $\alpha \in \R_t^n$ is 
$\Sh_{\alpha + \Lambda_{Z, M}}(\R^n_x / M)$, where $+\alpha$ just means we translate everything in the base direction by $\alpha$.  Evidently, this local system of categories descends to $\R_t^n / M$.  We denote the resulting monodromy representation by: 
\begin{equation}
    \label{equivariant monodromy}
    \widetilde{\mathcal{M}}: \Z^n / M \to Aut(\Sh_{\Lambda_{Z, M}}(\R^n /M))
\end{equation}

We may compose with $\rho_0$ to get a representation 
$\rho_0^* \widetilde{\mathcal{M}}: \Z^n / M \to Aut(\Sh_{\Lambda_{Z, M, 0}}(\R^n /M))$.  Observe: 
\begin{lemma} \label{noneq monodromy from eq monodromy}
$\rho_0^* \widetilde{\mathcal{M}} = - \mathcal{M}$    
\end{lemma}
\begin{proof}
    Recall from \eqref{reduction sequence}, \eqref{at gamma} that $\Lambda_{Z, M, 0}$ was defined as a certain symplectic reduction of $\Lambda_{Z, M}$; let us write $\Lambda_{Z, M, 0} := \Lambda_{Z, M}|_0$.  Observe that in the corresponding notation, we have $$(\alpha + \Lambda_{Z, M})|_0 = \Lambda_{Z, M, -\alpha}.$$ 
    Thus, pullback of the fiber categories along $M_\R / M \to \R^n /  M$ determines an isomorphism between the local systems  of categories on $\R^n / M$ defining $\widetilde{\mathcal{M}}$ and $-\mathcal{M}$. 
\end{proof}

Thus it will suffice to compute $\widetilde{M}$. 

\begin{proposition} \label{weak intertwining}
    For $\eta, \nu \in \Z^n/M$, there is an isomorphism
    $\mathcal{M}(\eta) (\Xi(\mathcal{O}(\mathcal{\nu}))) \cong \Xi(\mathcal{O}(\eta + \nu)).$
\end{proposition}

\begin{proof}
When $n = 1$, $Z = \emptyset$ and $M = 0$, i.e. for $\A^1 / \G_m$, this follows from Lemma \ref{a1 monodromy intertwining}.  We arrive at result for $\A^1 / \G_m^n$ by taking products. 
Now consider the diagram 

\[\begin{tikzcd}
	{\mathbb{A}^n/\mathbb{G}_{m}^n} & {(\mathbb{A}^n \setminus Z)/\mathbb{G}_m^n} \\
	{\mathbb{A}^n/G} & {(\mathbb{A}^n \setminus Z)/G}
	\arrow["{\iota_Z}"', hook', from=1-2, to=1-1]
	\arrow["{\pi_G}", from=2-1, to=1-1]
	\arrow["{\pi_G}"', from=2-2, to=1-2]
	\arrow["{\iota_Z}", hook', from=2-2, to=2-1]
\end{tikzcd}\]
The following functoriality is immediate from the construction of \cite{shende-toric}:

\begin{equation} \label{commuting diagram for monodromy} 
\begin{tikzcd}
	{\mathrm{Sh}_{\Lambda_{\emptyset, 0}}(\mathbb{R}^n)} &&& {\mathrm{Sh}_{\Lambda_{Z,0}}(\R^n)} \\
	& {\mathrm{Coh}(\mathbb{A}^n/\mathbb{G}_{m}^n)} & {\mathrm{Coh}((\mathbb{A}^n \setminus Z)/\mathbb{G}_m^n)} \\
	& {\mathrm{Coh}(\mathbb{A}^n/G)} & {\mathrm{Coh}((\mathbb{A}^n \setminus Z)/G)} \\
	{\mathrm{Sh}_{\Lambda_{\emptyset, M}}(\R^n/M)} &&& {\mathrm{Sh}_{\Lambda_{Z,M}}(\R^n/M)}
	\arrow[from=1-1, to=1-4]
	\arrow["{=}"{description}, no head, from=1-1, to=2-2]
	\arrow["{p_{M!}}", from=1-1, to=4-1]
	\arrow["{=}"{description}, no head, from=1-4, to=2-3]
	\arrow["{p_{M!}}"', from=1-4, to=4-4]
	\arrow["{\iota_{Z}^*}", from=2-2, to=2-3]
	\arrow["{\pi_G^*}"', from=2-2, to=3-2]
	\arrow["{\pi_G^*}", from=2-3, to=3-3]
	\arrow["{\iota_{Z}^*}"', from=3-2, to=3-3]
	\arrow["{=}"{description}, no head, from=3-2, to=4-1]
	\arrow[from=4-1, to=4-4]
	\arrow["{=}"{description}, no head, from=4-4, to=3-3]
\end{tikzcd}
\end{equation}
Here, the outer horizontal arrows are `stop removal' functors, i.e. adjoints to the natural inclusion of categories going in the other direction (recall $\Lambda_{Z, M} \subset \Lambda_{\emptyset, M}$).  
All the $\mathcal{O}(\eta)$ on $\Coh((\A^n \setminus Z)/G)$ are images of the corresponding objects in $\Coh(\A^n / \G_m^n)$.  Thus, to establish the proposition, it suffices to check it for $\Coh(\A^n / \G_m^n)$ and to argue that the outer square commutes with
$\widetilde{\mathcal{M}}$.  This final commutativity
follows from the fact that we could make the the outer square using $\widetilde{\Lambda}$ instead of $\Lambda$. 
This completes the proof. 
\end{proof}

Lemma \ref{noneq monodromy from eq monodromy} and Proposition \ref{weak intertwining} are already enough for our geometric applications.  However, one may wish to know the a priori stronger assertion

\begin{theorem} \label{strong intertwining}
    $\Xi_0$ interwines $\mathcal{M}$ and $\mathcal{P}$. 
\end{theorem}
\begin{proof}
    The remaining thing to show is that $\Xi^* \widetilde{\mathcal{M}}$ is given by tensor product with line bundles; having done so, which line bundles these are is determined by  Proposition \ref{weak intertwining}.  When $(\A^n \setminus Z) / G$ is a smooth proper toric variety, we may appeal (as do \cite{GHHPS}) to \cite{bondal2001reconstruction} to learn that in fact the only autoequivalences of this category are tensor by (possibly shifted) line bundles.  
    We give a proof for arbitrary smooth toric DM stacks in the next section. 
\end{proof}

\begin{remark}
    It would be interesting to compare our monodromies 
    with those studied in \cite{hanlon-monodromy} and \cite{Spenko-VanDenBergh}.  The comparison to the former could plausibly be done by adapting the skeletal calculations of \cite{Gammage-Shende-very-affine, Zhou-skel} to the situation of varying coefficients in the superpotential.
\end{remark}

\section{Monoidal considerations}

Note that any quotient of $\R^n$ by a subgroup  $M \subset \Z^n$ is naturally a group, with structure $$+ : \R^n/M \times \R^n / M \to \R^n/M.$$ Its category of sheaves carries a monoidal structure from convolution: 
\begin{eqnarray*}
    \star: \Sh(\R^n / M) \otimes \Sh(\R^n/M) & \to & \Sh(\R^n / M) \\
    F \boxtimes G & \mapsto & +_{!} (F \boxtimes G) 
\end{eqnarray*}
Recall that the skyscraper sheaf at the identity is the monoidal unit for convolution. 

Let us observe:

\begin{lemma} \label{O}
    Let $(\A^n \setminus Z)/G)$ be a  DM toric stack with proper coarse moduli space; i.e. assume the fan is complete.  
    Then any monoidal functor $(\Coh(\A^n \setminus Z)/G), \otimes) \to (\Sh_{\Lambda_{Z, M, 0}}(M_\R/M), \star)$ sends $\mathcal{O}$ to $\Z_0$, the skyscraper sheaf at $0$. 
\end{lemma}
\begin{proof}
    We must check that the skyscraper sheaf $\Z_0$ remains the monoidal unit after imposing the microsupport condition $\Lambda_{Z, M, 0}$.  This will certainly be true so long as it remains an element of the category, which is true so long as $\Lambda_{Z, M, 0}$ contains the conormal to $0$.  
    But the explicit description of $\Lambda_{Z, M, 0}$ in \cite{shende-toric} (i.e. the identification of it with the microsupport condition identified long earlier by \cite{FLTZ2}) shows that the fiber of $\Lambda$ over $0$ is just the fan itself.   
\end{proof}

\begin{remark}
    That $\Phi_{FLTZ}(\mathcal{O}) = \Z_0$ is not immediately obvious from 
    the definition of the \cite{FLTZ2} functor, nor from our later construction \cite{shende-toric}.  In fact, for arbitrary (not proper) toric varieties, I do not know an explicit characterization. 
\end{remark}

The mirror functor $\Xi_{FLTZ}$ constructed in \cite{FLTZ2} was shown there to be monoidal, hence Lemma \ref{O} can be applied to $\Xi_{FLTZ}$.  
We however are going to want to know the conclusion of Lemma \ref{O} compatibly with the conclusion of Proposition \ref{weak intertwining}.  Thus we should either verify $\Xi_{FLTZ} \cong \Xi_0$, or show directly that $\Xi_0$ is monoidal.  We do the first:

\begin{proposition}
    $\Xi_0 \cong \Xi_{FLTZ}$ for smooth toric DM stacks. 
\end{proposition}
\begin{proof}
    The mirror functor $\Xi_{FLTZ}$ of \cite{FLTZ2} agrees with $\Xi = \Xi_0$ for $\A^n$, essentially by definition of the latter.  It is straightforward to see that, for $\A^N/G$ with $G \subset \G_m^n$ finite, the stacky prescription in \cite{FLTZ3} is related to that in \cite{FLTZ2} by de/equivariantization, hence also matches the corresponding prescription of $\Xi = \Xi_0$. 

    In general, the map $\Xi_{FLTZ}$ is characterized by the following property: given a toric inclusion $\A^n / G \to T$, the map $\Xi_{FLTZ}$ intertwines the pushforward of sheaves on both sides; or, adjointly, it intertwines the pullback of coherent sheaves with stop removal of the extraneous parts of the microsupport condition.  (Indeed, the `shard sheaves' in terms of which $\Xi_{FLTZ}$ is defined in \cite{FLTZ2, FLTZ3} are pushed forward from the mirror to some $\A^n/G$.) 

    Thus we need to verify that  $\Xi_0$ also has this behavior for toric inclusions of $\A^n/G$. 
    Here note that the affine toric charts of $(\A^n \setminus Z) / G$ are necessarily of the form $(\A^n \setminus Z')/ G$ for some larger collection of coordinate hyperplanes $Z' \supset Z$.  
    Now the desired commutativity for the functor $\Xi$ is the natural generalization of the bottom trapezoid of \eqref{commuting diagram for monodromy}.  We may pass it to $\Xi_0$ since, by the key non-characteristicity result established in \cite{shende-toric}, we may remove the stops equivalently before or after the restriction $\rho_0$ to the fiber. 
\end{proof}

\begin{proof}[Proof of Theorem \ref{strong intertwining} for smooth toric DM stacks]
    The same estimate showing that convolution gives a monoidal structure on
    $\Sh_{\Lambda_{Z, M, 0}}(M_\R/M)$ similarly shows that 
    $\Sh_{\Lambda_{Z, M, \gamma}}(M_\R/M)$ is a module for said monoidal structure. 
    Thus the monodromy $\mathcal{M}$ in fact acts by endomorphisms of  
    $\Sh_{\Lambda_{Z, M, 0}}(M_\R/M)$ as a $(\Sh_{\Lambda_{Z, M, 0}}(M_\R/M), \star)$-module.  For any symmetric monoidal category, such endomorphisms are necessarily
    the product with an invertible object. 

    Since $\Xi_0 \cong \Xi_{FLTZ}$ is monoidal, we see that $\Xi_0^* \mathcal{M}$ is given by tensor products with invertible objects, hence (possibly shifted) line bundles. 
\end{proof}
 
\section{The sheaf-Fukaya comparison and compact Lagrangians}

Let $W$ be a Weinstein manifold, 
and $\mathfrak{c}(W)$ its core.  
Recall that in \cite{shende-microlocal, Nadler-Shende}, we defined, for any compact conic subset $K$ of $W$, a category $\mu sh(K)$.  These categories grow under inclusion: for $K \subset K'$ one has $\mu sh(K) \hookrightarrow \mu sh(K')$. 

In \cite{GPS3} we showed that if $\mathfrak{c}(W)$ is the core of $W$, then
\begin{equation} \label{gps equivalence} \Fuk(W) \cong \mu sh(\mathfrak{c}(W)),
\end{equation}
and that moreover if $p \in \mathfrak{c}(W)$ is a smooth Lagrangian point,
and $D_p \subset W$ is an eventually conic closed Lagrangian disk transverse to $p$, 
then $D_p$ is sent to a microstalk functor at $p$. However, we did not discuss how the equivalence \eqref{gps equivalence}  acted on other objects.  For our purposes now, we will need to understand its action on compact Lagrangians.

Recall that for any (singular) conic Lagrangian 
$\Lambda$, the category $\mu sh(\Lambda)$ is generated by microstalk functors at smooth points of $\Lambda$.  In particular, $\mu sh(\Lambda)^{pp}$ is the subcategory of objects which have perfect microstalks.  
For cotangent bundles, we know  that the cotangent fiber generates and its endomorphisms are the chains on the based loop space of the base manifold (originally from \cite{Abouzaid-cotangent}; or, logically independently, from either \cite{GPS2} or \cite{GPS3}).  It follows that all of $\Fuk(T^*M)^{pp}$ is realized by perfect local systems on $M$.  Thus  the equivalence
\eqref{gps equivalence} carries  
$\mu sh(\mathfrak{c}(T^*M) = M)^{pp}$ to objects arising from perfect local systems on $M \subset T^*M$.   

Here we want to know the following generalization: 

\begin{proposition} \label{annoying obvious fact}
    Let $L \subset \mathfrak{c}(W)$ be a smooth compact Lagrangian. 
    Then the equivalence of \cite{GPS3} identifies $\mu sh(L)^{pp}$ with the subcategory of $\Fuk(W)$ on objects given by a finite rank local system on $L$.  
    
    More generally, for any smooth compact $L\subset W$ and its Legendrian lift $\widetilde{L} \subset W \times \R$, the composition of \cite{GPS3} with the 
fully faithful microlocal nearby cycle functor of \cite{Nadler-Shende} identifies 
    $\mu sh (\widetilde{L})^{pp}$ with  the subcategory of $\Fuk(W)$ on objects given by a finite rank local system on $L$.  
\end{proposition}

The remainder of this section serves to prove this proposition.  For the results elsewhere in this article, we need only the case of $L \subset \mathfrak{c}(W)$.  

\vspace{2mm}
As the equivalence of \cite{GPS3} does not speak directly of compact Lagrangians, to prove the proposition it will be convenient for the proposition to be able to express ``objects given by a finite rank local system on $L$'' without speaking of objects. 

To do this we recall from \cite{GPS2} Sylvan's construction of a `Viterbo restriction' functor.  Let $W$ be a Liouville, and $V \subset W$ a Liouville subdomain.  Consider the stopped Liouville manifold 
$$[W \to V] := (W \times \C, W \times -\infty \sqcup V \times \infty)$$  The covariant inclusion functors from \cite{GPS1} give maps 
$$ \Fuk(W) \rightarrow \Fuk([W \to V]) \hookleftarrow \Fuk(V)$$
where the full faithfulness of the right map follows from geometric considerations of stopping of wrapping.  When $V$ is Weinstein, this map is moreover an equivalence, and thus we obtain a map 
$\Fuk(W) \to \Fuk(V)$.  Here we merely observe: 

\begin{lemma} \label{viterbo versus inclusion}
    The pullback $\Fuk(V)^{pp} \to \Fuk(W)^{pp}$ along the Sylvan-Viterbo restriction map sends a compact Lagrangian (with local system etc.) to ``the same'' compact Lagrangian. 
\end{lemma}
\begin{proof}
    Fix such a Lagrangian (and associated data) $L \subset V$.  Let $L_V \in \Fuk(V)$ and $L_W \in \Fuk(W)$ be the associated objects.  To prove the proposition, it suffices to show that the images of these two objects are isomorphic in $\Fuk([W \to V])$.  But this is immediate from the definition of the covariant inclusion, as both are (isotopic to) the same object on $L \times i\R$. 
\end{proof}

Thus we see that ``objects given by a finite rank local system on $L$''  are characterized as the image of 
$\Fuk(T^*L)^{pp} \to \Fuk(W)^{pp}$
where the map is the pullback along Sylvan-Viterbo restriction. 

\begin{theorem} \label{viterbo commutes with GPS}
    Consider an inclusion of Weinstein domains $V \subset W$ such that $\mathfrak{c}(V) \subset \mathfrak{c}(W)$.  Then the following diagram commutes, where the top arrow is adjoint to Sylvan-Viterbo (recall we ind-complete all our categories) and the bottom is inclusion of microsheaves. 
\[\begin{tikzcd}
	{\mathrm{Fuk}(W)} & {\mathrm{Fuk}(V)} \\
	{\mu sh(\mathfrak{c}(W))} & {\mu sh(\mathfrak{c}(V))}
	\arrow[Rightarrow, no head, from=1-1, to=2-1]
	\arrow[from=1-2, to=1-1]
	\arrow[Rightarrow, no head, from=1-2, to=2-2]
	\arrow[from=2-2, to=2-1]
\end{tikzcd}\]
    More generally, without assuming $\mathfrak{c}(V) \subset \mathfrak{c}(W)$, the same is true with the bottom arrow being given by the `microlocal nearby cycle' functor of \cite{Nadler-Shende}. 
\end{theorem}
\begin{proof}
    The following diagram commutes, 
    where all horizontal maps are either the \cite{GPS1} covariant inclusion functor or the restriction of microsheaves: 
\[\begin{tikzcd}
	{\mathrm{Fuk}(W)} & {\mathrm{Fuk}([W \to V])} & {\mathrm{Fuk}(V)} \\
	{\mu sh(\mathfrak{c}(W))} & {\mu sh(\mathfrak{c}([W \to V]))} & {\mu sh(\mathfrak{c}(V))}
	\arrow[Rightarrow, no head, from=1-1, to=2-1]
	\arrow[from=1-2, to=1-1]
	\arrow["\sim", from=1-2, to=1-3]
	\arrow[Rightarrow, no head, from=1-2, to=2-2]
	\arrow[Rightarrow, no head, from=1-3, to=2-3]
	\arrow[from=2-2, to=2-1]
	\arrow[from=2-2, to=2-3]
\end{tikzcd}\]
Thus it suffices to show that the map 
$\mu sh(\mathfrak{c}([W \to V])) \to \mu sh(\mathfrak{c}(V))$ admits a section
so that the further composition 
$\mu sh(\mathfrak{c}(V)) \to \mu sh(\mathfrak{c}([W \to V]))  \to \mu sh(\mathfrak{c}(W))$ agrees with the map  $\mathfrak{c}(V) \to \mathfrak{c}(W)$.
When $\mathfrak{c}(V) \subset \mathfrak{c}(W)$, 
such a section is given by the closed embedding of $\mathfrak{c}(V) \times \R \subset \mathfrak{c}([W \to V])$, along with the stabilization isomorphism
$\mu sh(\mathfrak{c}(V)) \cong \mu sh(\mathfrak{c}(V) \times \R)$.

More generally, we may compose with the microlocal nearby cycle -- here we need the relative version \cite[Cor. 9.18]{Nadler-Shende} -- to get a fully faithful embedding of the bottom row into 
$$\mu sh(\mathfrak{c}(W)) \xleftarrow{\sim} \mu sh(\mathfrak{c}(W) \times \R) \xrightarrow{\sim} \mu sh(\mathfrak{c}(W)).$$  Composition with these equivalences again provides the desired section. 
\end{proof}

\begin{proof}[Proof of Proposition \ref{annoying obvious fact}]
Take $V = D^*L$ a disk bundle around $L$, and apply Lemma \ref{viterbo commutes with GPS}, restricted to pseudoperfect objects.  Thus we have: 
\[\begin{tikzcd}
	{\mathrm{Fuk}(W)^{pp}} & {\mathrm{Fuk}(T^*L)^{pp}} \\
	{\mu sh(\mathfrak{c}(W))^{pp}} & {\mu sh(L)^{pp}}
	\arrow[Rightarrow, no head, from=1-1, to=2-1]
	\arrow[from=1-2, to=1-1]
	\arrow[Rightarrow, no head, from=1-2, to=2-2]
	\arrow[from=2-2, to=2-1]
\end{tikzcd}\]
Now Lemma \ref{viterbo versus inclusion} reduces the result to the special case of cotangent bundles, which we have already seen to hold.  
\end{proof}

\begin{remark}
    Proposition \ref{annoying obvious fact} can be deduced similarly but with somewhat less fuss from \cite[Prop. 11.2]{GPS3} in the case that the cobordism $W \setminus D^*L$
    is also Weinstein.  This is presumably the case in our applications, but is not immediately obviously so, because our $L$ is not a minimum of the  Morse function. 
\end{remark}

\section{Proof of Theorem \ref{fiber enough lagrangians}}
Let $\mathbf{T}$ be a smooth toric stack, $\partial \mathbf{T}$ its toric boundary, $\Lambda_{\mathbf{T}} \subset T^*(M_\R / M)$ the corresponding FLTZ skeleton, and $\partial_\infty \Lambda_{\mathbf{T}} \subset S^* (M_\R / M)$ its ideal Legendrian boundary.  Recall from \cite{Gammage-Shende-very-affine, Zhou-skel} that the Hori-Vafa Landau-Ginzburg model mirror to $\mathbf{T}$ carries the structure of a Weinstein pair $((\C^*)^n, F_{\mathbf{\partial T}})$ with $(\Lambda_{\mathbf{T}}, \partial_\infty \Lambda_{\mathbf{T}} )$ as its relative skeleton, and from \cite{Gammage-Shende-very-affine} that there is a commuting diagram 
\begin{equation} \label{GS1 diagram}
\begin{tikzcd}
	{\mathrm{Fuk}((\mathbb{C}^*)^n, F)} & {\mathrm{Sh}_{\Lambda_{\mathbf{T}}}(M_{\mathbb{R}}/M)} & {\mathrm{Coh}(\mathbf{T})} \\
	{\mathrm{Fuk}(F)} & {\mu sh(\partial_\infty \Lambda_{\mathbf{T}})} & {\mathrm{Coh}(\partial \mathbf{T})}
	\arrow["\shortmid"{marking}, from=1-1, to=2-1]
	\arrow[Rightarrow, no head, from=1-2, to=1-1]
	\arrow["\mu"', from=1-2, to=2-2]
	\arrow["\sim"', from=1-3, to=1-2]
	\arrow["{i^*}", from=1-3, to=2-3]
	\arrow[Rightarrow, no head, from=2-2, to=2-1]
	\arrow["\sim"', from=2-3, to=2-2]
\end{tikzcd}
\end{equation}
Here, the map indicated on Fukaya categories is the adjoint to the \cite{GPS1} covariant inclusion (recall we have everywhere passed to ind completions).  The corresponding maps on categories of pseudoperfect module categories all go in the directions indicated.  In particular, by commutativity of the left square, $\mathcal{O}_{\partial \mathbf{T}} = i^*  \mathcal{O}_{\mathbf{T}}$ is sent by mirror symmetry to the microlocal restriction of the skyscraper sheaf at $0 \in M_\R / M$.  This is some microsheaf supported along the cosphere over $0$.  Per Proposition \ref{annoying obvious fact}, this is carried by the \cite{GPS3} Fukaya/sheaf dictionary to some object in $\Fuk(F)$ supported on the same sphere. 

Note that by \cite[Prop. 13]{shende-toric}, the family of sheaf categories giving rise to the monodromy $\mathcal{M}$ arises as the relative skeleta of a family of Weinstein pairs; it follows that $\mathcal{M}$ is carried by the Fukaya/sheaf dictionary to the endomorphism induced by symplectic parallel transport.  Indeed, recall that if $W_t$ for $t\in [0,1]$ is a family of Weinstein manifolds, then the symplectic parallel transport $\Fuk(W_0) \to \Fuk(W_1)$ determines a `mapping cylinder' $\widetilde{W}$ which is an exact symplectic fibration over $T^*[0,1]$, and can be regarded as a Weinstein pair relative copies of $W_0$ and $W_1$ over $0, 1$.  Then the symplectic parallel transport is given by the composition $$\Fuk(W_0) \xleftarrow{\sim} \Fuk(\widetilde{W}; W_0 \sqcup W_1) \xrightarrow{\sim} \Fuk(W_1).$$  Carrying this latter picture across to sheaves is how we defined the monodromy $\mathcal{M}$ on $\Sh_{\Lambda_{\mathbf{T}}}(M_\R/M)$.  

The monodromy of symplectic parallel transport along a family  of Weinstein pairs evidently commutes with the \cite{GPS1} covariant inclusion.  So commutativity of the diagram along with Proposition \ref{weak intertwining} shows that the images of line bundles under $i^*$ are carried to the images of the sphere $S^*_0 (M_\R/M)$ under symplectic parallel transport.  
$\square$

\section{Proof of Theorem \ref{fanifold enough Lagrangians}}
    Fix the closed fanifold $\Phi$ and denote the \cite{Gammage-Shende-large-volume} mirror symmetry as 
    \begin{equation}
        \Xi: \Coh(\mathbf{T}(\Phi)) \xrightarrow{\sim} \mu sh(\mathbb{L}(\Phi)) = \Fuk(\mathbf{W}(\Phi))
    \end{equation}

    Let us recall some of the setup of \cite{Gammage-Shende-large-volume}.  
    For a stratum $\sigma \subset \Phi$, we will write $\mathbf{T}(\sigma)$ for the corresponding toric stack.  When $\sigma \subset \overline{\tau}$, we have a closed inclusion of $\mathbf{T}(\tau) \subset \partial \mathbf{T}(\sigma)$.  The space $\mathbf{T}(\Phi)$ is built by gluing along these inclusions. 

    For a stratum $\sigma$, we write $r(\sigma)$ for the set of rays in the corresponding fan.  Note $r(\sigma)$ appears in the Cox presentation as $$\mathbf{T}(\sigma) = (\A^{r(\sigma)} \setminus Z(\sigma)) / G(\sigma).$$ 

    Note that  
    $\sigma \subset \overline{\tau}$ gives an inclusion $r(\tau) \subset r(\sigma)$.  We write $\Z^{r(\sigma)} \twoheadrightarrow \Z^{r(\tau)}$ for the corresponding projection.  Let $M(\sigma)$ be the character lattice of the torus of $\mathbf{T}(\sigma)$.  Recall that there is an embedding $M(\sigma) \subset \Z^{r(\sigma)}$ and $Pic(\mathbf{T}(\sigma)) = \Z^{r(\sigma)}/M(\sigma)$.  Restriction of line bundles to toric boundary strata is correspondingly given as:
    \begin{equation} \label{pic restriction} \Z^{r(\sigma)}/M(\sigma) \to \Z^{r(\tau)}/M(\tau).
    \end{equation}
    We regard these maps as the generization maps for a constructible sheaf of abelian groups on $\Phi$; we will denote this sheaf as $\Z^r / M$.  

\begin{lemma}
    There is an exact sequence 
        \begin{equation}\label{fanifold pic} 0 \to H^1(\Phi, \G_m) \to Pic(\mathbf{T}(\Phi))  \xrightarrow{\mathrm{deg}} H^0(\Phi, \Z^r/M) \to 0.
    \end{equation}
\end{lemma}
\begin{proof}
        A line bundle on $\mathbf{T}(\Phi)$ 
    in particular restricts to a line bundle on each $\mathbf{T}(\sigma)$.  This determines a surjective map 
    $Pic(\mathbf{T}(\Phi)) \to H^0(\Phi, \Z^r/M)$.  The kernel concerns the gluing of the line bundles on higher codimension strata; this is a $\G_m$ choice on the toric varieties associated to codimension one strata of $\Phi$, subject to the cocycle condition in codimension two.
\end{proof}
    As for any section of any constructible sheaf,  $\deg(\mathcal{L})$ is characterized by its stalks at minimal strata, or what is the same, the restriction
    of $\mathcal{L}$
     to the irreducible components of $\mathbf{T}(\Phi)$.  
    We will also write $Pic_0(\mathbf{T}(\Phi))$ for the kernel of $\deg$.
  
     We will write 
    $$\Lambda_\gamma(\sigma) := \Lambda_{Z(\sigma), M(\sigma), \gamma} \subset T^*(M(\sigma)_\R/M(\sigma))$$
    Recall that the core $\mathbb{L}(\Phi)$ of $\mathbf{W}(\Phi)$ is a union over point strata $\sigma$ of the $\Lambda_0(\sigma)$.    
    The mirror symmetry of \cite{Gammage-Shende-large-volume} carries the following structure maps: 
\begin{equation} \label{large volume mirror}
\begin{tikzcd} 
	{\mathrm{Fuk}(\mathbf{W}(\Phi))} & {\mu sh(\mathbb{L}(\Phi))} & {\mathrm{Coh}(\mathbf{T}(\Phi))} \\
	& {\mathrm{sh}_{\Lambda_0(\sigma)}(M(\sigma)_{\mathbb{R}}/M(\sigma))} & {\mathrm{Coh}(\mathbf{T}(\sigma))}
	\arrow[Rightarrow, no head, from=1-2, to=1-1]
	\arrow[from=1-2, to=2-2]
	\arrow["\Xi"', from=1-3, to=1-2]
	\arrow[from=1-3, to=2-3]
	\arrow["\Xi_\sigma"', from=2-3, to=2-2]
\end{tikzcd}
\end{equation}    
    All horizontal maps in the diagram are isomorphisms.  In \cite{Gammage-Shende-large-volume}, the top isomorphism $\Xi$ was defined by taking a colimit over $\sigma$ of the $\Xi_\sigma$.  

    Recall there is a map $\mathbb{L}(\Phi) \to \Phi$, and a section $s$ whose image is the the union of the cotangent-fiber-at-zero pieces of $\Lambda_{0}(\sigma)$ over the point strata $\sigma$.  

    \begin{lemma} \label{degree zero lines represented}
        $\Xi$ carries $\mathrm{Pic}_0(\mathbf{T}(\Phi))$ to rank one microsheaves supported on $s(\Phi)$.
    \end{lemma}
    \begin{proof}
        By definition, $Pic_0$ consists of line bundles which restrict to $\mathcal{O}$ on every irreducible component.  Now the commutativity of \eqref{large volume mirror} reduces the result to Lemma \ref{O}. 
    \end{proof}

    By Proposition \ref{annoying obvious fact}, microsheaves on $s(\Phi)$ correspond to objects in the Fukaya category given by rank one local systems on the Lagrangian $s(\Phi)$. 

    \vspace{2mm}

    Recall that $\partial_\infty \Lambda_0(\sigma)$ naturally decomposes as a union of (products with disks) of the $\Lambda_0(\tau)$ mirror  to the components $\mathbf{T}(\tau)$ of $\partial \mathbf{T}(\sigma)$ (see e.g. \cite{Gammage-Shende-very-affine}[Cor. 4.3.2], \cite{Gammage-Shende-large-volume}[Lemma 4.16]).  One sees by the same arguments that  $\partial \Lambda_{\gamma(\sigma)}(\sigma)$ is covered by (disks times) $\Lambda_{\gamma(\tau)}(\tau)$, 
    where $\gamma(\sigma)$ and $\gamma(\tau)$ are related by the map 
    $\R^{r(\sigma)}/M(\sigma)_\R \to \R^{r(\tau)}/M(\tau)_\R.$
    obtained by extending scalars in \eqref{pic restriction}.  
    
    A system of $\gamma(\sigma)$ with such generization properties is precisely some
    $$\gamma \in H^0(\Phi, (\R^r/ M_\R) / (\Z^r / M)).$$
    Given such $\gamma$, we may construct some Weinstein manifold  $\mathbf{W}_\gamma(\Phi)$ with core
    $\mathbb{L}_\gamma(\Phi)$ built from gluing the $\Lambda_{\gamma}(\sigma)$
    by repeating verbatim the arguments of \cite[Section 4]{Gammage-Shende-large-volume}.  These arguments described a procedure of iterated handle attachment, and varying $\gamma$ simply varies the handles at each stage by Legendrian isotopy; it follows that $\mathbf{W}_\gamma(\Phi)$ is a family of Weinstein manifolds.  

    The fundamental group of the space of $\gamma$'s is: 
    \begin{equation} \label{fanifold loops}
        \pi_1 ( H^0(\Phi, (\R^r/ M_\R) / (\Z^r / M))) = H^0(\Phi, \Z^r/M) 
    \end{equation}
    Symplectic parallel transport determines a map: 
    $$\mathcal{M}: H^0(\Phi, \Z^r/M) \to Aut(\Fuk(\mathbf{W}(\Phi))).$$
    
    We will want a microsheaf formulation.  Consider a loop $\gamma: S^1 \to H^0(\Phi, (\R^r/ M_\R) / (\Z^r / M))$.  
    Let $\widetilde{\mathbf{W}} \to T^*S^1$ be the space whose symplectic parallel transport realizes the monodromy; by construction its core is some $p: \widetilde{\mathbb{L}} \to S^1$ with fiber $\mathbb{L}_{\gamma(t)}(\Phi)$ over $t \in S^1$.  
    We consider the sheaf of categories $\mu sh_{\widetilde{\mathbb{L}}}$.  We claim $p_* \mu sh_{\widetilde{\mathbb{L}}}$ is locally constant.  Probably one can see this just with \cite{Nadler-Shende}; we just invoke \cite{GPS3} and note that above any open interval of $S^1$, the section category is equivalent to the Fukaya category of (the product with an interval of) $\Fuk(\mathbb{W}(\Phi))$.   In addition we claim that canonically $\mu sh_{\widetilde{\mathbb{L}}}|_{p^{-1}(t)} \xrightarrow{\sim} \mu sh_{\mathbb{L}_{\gamma(t)}(\Phi)}$.  Indeed, there would always be a map in this direction with the target possibly having larger (micro)support.  The control on the microsupport and the fact that the map are equivalences can be checked locally, hence on our standard cover by $\Lambda_{\gamma(t)}$, where the result is  \cite[Thm. 14]{shende-toric}.  Now we may use local constancy over $S^1$ to directly define 
    $$\mathcal{M}: H^0(\Phi, \Z^r/M) \to Aut(\mu sh(\mathbb{L}(\Phi)))$$
    \begin{lemma} \label{monodromy restriction}
        Consider $F \in \mu sh(\mathbb{L}(\Phi))$. 
        We have the compatibility:
        $$\mathcal{M}(\gamma)(F)|_{\Lambda_0(\sigma)} = 
        \mathcal{M}(\gamma(\sigma))(F|_{\Lambda_0(\sigma)})$$
    \end{lemma}
    \begin{proof}
        By the above discussion, we may uniquely continue $F$ along $S^1$.  On the other hand, we know, already from the construction the monodromies along $\Lambda_{\gamma(\sigma)}(\sigma)$, that we may already uniquely continue the restriction of $F$.  The restriction of the continuation must be the continuation of the restricion. 
    \end{proof}

    We write $\Xi^* \mathcal{M}$ for the pullback of the monodromy representation along mirror symmetry. 

    \begin{corollary} \label{degree monodromy}
    Let $\mathcal{L}$ be a line bundle on $\mathbf{T}(\Phi)$.  Then: 
    \begin{equation} \deg (\mathcal{L})- \deg (\Xi^* \mathcal{M}(\gamma)(\mathcal{L}) )  = [\gamma] \in   H^0(\Phi, \Z^r/M) \end{equation}        
    \end{corollary}
    \begin{proof}
        Follows from Lemma \ref{monodromy restriction}, the commutativity of \eqref{large volume mirror}, and Proposition \ref{weak intertwining}.  
    \end{proof}

    Finally, let us complete the proof the Theorem \ref{fanifold enough Lagrangians}.  Let $\mathcal{L}$ be a line bundle on $\mathbf{T}(\Phi)$.  Choosing some loop $\gamma$ with $[\gamma] = \deg(\mathcal{L})$, we see from Corollary \ref{degree monodromy} that $(\Xi^* \mathcal{M}(\gamma))^{-1}(\mathcal{L})$ has degree zero.  It follows from 
    Lemma \ref{degree zero lines represented} and Proposition \ref{annoying obvious fact} that $\Xi$ carries  $\Xi^{-1} \mathcal{M}(\gamma) \Xi (\mathcal{L})$ to a local system on the compact Lagrangian $s(\Phi)$.  Applying now the symplectic parallel transport along $\gamma^{-1}$, we that $\Xi(\mathcal{L})$ can be represented by a local system on the symplectic parallel transport of $s(\Phi)$.  $\square$

\begin{remark}
    In the above proof, we are never forced to study explicitly the gluing of microsheaves or line bundles along higher codimension strata; instead we are able to get away with independently considering their restrictions to the top dimensional toric strata and their mirrors.  (We do use the geometry of the higher codimension strata when constructing the monodromies.)  Because of this, we never needed the analogue of \eqref{GS1 diagram} for $\gamma\ne0$.  
\end{remark}

\bibliographystyle{plain}
\bibliography{refs}
\end{document}